\documentclass[12pt]{amsart}

\usepackage[dvips]{graphics}
\usepackage{color, amssymb, amsbsy, amsmath, amsfonts, amssymb, 
amscd, epsfig, times, graphics}

\newtheorem{theorem}{Theorem}[section]

\newtheorem{lemma}[theorem]{Lemma}

\theoremstyle{definition}
\newtheorem{definition}[theorem]{Definition}

\numberwithin{equation}{section}

\allowdisplaybreaks

\begin{document}

\begin{abstract}
This paper studies the complex Monge-Amp{\`e}re equations for $\mathcal F$-plurisubharmonic functions in bounded $\mathcal F$-hyperconvex domains. We give sufficient conditions for this equation to solve for measures with a singular part.
\end{abstract}

\title[Complex Monge-Amp\`ere equations]{Complex Monge-Amp\`ere equations for plurifinely plurisubharmonic functions}

%\dedicatory{Dedicated to Prof. Le Mau Hai on the occasion of his 70th birthday}

\author[N. X. Hong]{Nguyen Xuan Hong}
\address{Department of Mathematics, Hanoi National University of Education, 136 Xuan Thuy Street, Cau Giay District, Hanoi, Vietnam}\email{hongnx@hnue.edu.vn}

\author[H. V. Can]{Hoang Van Can}
\address{Department of Basis Sciences,  University of Transport Technology, 54 Trieu Khuc, Thanh Xuan District, Hanoi, Vietnam}\email{canhv@utt.edu.vn}

\author[N. T. Lien]{Nguyen Thi Lien$^\dag$}
\address{Department of Mathematics, Hanoi National University of Education, 136 Xuan Thuy Street, Cau Giay District, Hanoi, Vietnam} \email{ntlien@hnue.edu.vn}

\author[P. T. Lieu]{Pham Thi Lieu$^\ddag$}
\address{Department of Basis Sciences and Foreign Languages, People's Police University of Technology and Logistics, Bacninh, Vietnam\\ and Minh Khai, Hung Ha, Thai Binh, Vietnam}\email{ptlieu2@gmail.com}


%\thanks{$^\dag$This work was supported by  Hanoi National University of Education [grant number  SPHN21-05].  }
%\thanks{$^\ddag$Corresponding author. The authors would like to thank the referees for valuable remarks that helped to improve the exposition in this paper.}
\maketitle

%\tableofcontents

\newcommand{\di}[2]{d#1(#2)}

\section{Introduction and results}

The plurifinely topology $\mathcal F$ on a  Euclidean open set $\Omega \subset \mathbb{C}^n$  is the weakest topology that makes all plurisubharmonic functions on $\Omega$ continuous. Notions pertaining to the plurifinely topology are indicated with the prefix $\mathcal F$ to distinguish them from notions pertaining to the Euclidean topology on $\mathbb C^n$. 
The notion $\mathcal F$-plurisubharmonic functions in  $\mathcal F$-open subsets $\Omega$ of $\mathbb C^n$ and basic properties of these functions are introduced in \cite{K03}. Recall that an  $\mathcal F$-upper semicontinuous function $u$ defined on an $\mathcal F$-open set $\Omega$  is $\mathcal F$-plurisubharmonic   if  for every complex line $l$ in $\mathbb C^n$, the restriction of $u$ to any $\mathcal F$-component of the finely open subset $l \cap \Omega$  of $l$ is either finely subharmonic or $\equiv -\infty$ (see \cite{ W12}). 

When $\Omega$ is Euclidean open, the class of $\mathcal F$-plurisubharmonic functions is identical to the class of plurisubharmonic functions on $\Omega$ (see \cite{KFW11}).
The Monge-Amp\`ere operator of a smooth plurisubharmonic function $u$ can be defined as $$ (dd^c u)^n=n!4^n \det \bigg( \dfrac{\partial^2 u}{\partial z_j \partial \overline {z}_k}\bigg) dV_{2n},$$
 where $dV_{2n}$ is the volume form in $\mathbb C^n.$  In 1982, E. Bedford and B. A. Taylor \cite{BT82}  gave the definition of the complex Monge-Amp\`ere operator for the class of the locally bounded plurisubharmonic functions (also see \cite{BT76, BT87}).  
After that, the  Monge-Amp\`ere operator for finite $\mathcal F$-plurisubharmonic functions in  $\mathcal F$-domain is defined by M. El Kadiri and J. Wiegerinck \cite{KW14}.  They used the  
fact that  any finite $\mathcal F$-plurisubharmonic function $u$ on an $\mathcal F$-domain  can $\mathcal F$-locally at $z\in \Omega$ be written as $f-g$ where $f,g$ are bounded plurisubharmonic functions defined on a ball about $z$. Therefore, the non-polar part $NP(dd^cu)^n$   is  $\mathcal F $-locally defined by 
%of plurisubharmonic functions the Monge-Amp\`ere may be defined by multilinearity,
$$NP(dd^cu)^n:=\sum\limits_{p=0}^n\binom{n}{p} (-1)^p (dd^cf)^{n-p}\wedge (dd^c g)^p.$$
%Thus,The notion of the non-polar part $NP(dd^cu)^n$ of $\mathcal F$-plurisubharmonic functions is also introduced in \cite{KW14}. 
Recently, the second author and the fourth author studied the pluripolar part $P(dd^c u)^n$ of complex Monge-Amp\`ere measures of a $\mathcal F$-plurisubharmonic function $u$ defined in a bounded $\mathcal F$-hyperconvex domain (see \cite{CL22}).

% The complex Monge-Amp\`ere equation is the problem of finding a plurisubharmonic function $u$ such that $$(dd^cu)^n=\mu,$$ where $\mu$ is a non-negative measure. In \cite{BT76},   E. Bedford and B. A. Taylor proved that the complex Monge-Amp\`ere equation is solvable if. $$\mu=fdV_{2n} \text{ with } f\in C(\overline{\Omega}).$$   U. Cegrell in \cite{Ce2} showed that the problem is solvable if $\mu$ vanishes on pluripolar sets and $\Omega$ is hyperconvex domains. In \cite{ACCH}, P. \AA hag, U. Cegrell, R. Czy\. z and P. H. Hiep investigated the  problem for non-negative measures carried by a pluripolar set. Later, N. X. Hong et al  in \cite{Hong17,HC18b}  considered the problem in  $\mathcal F$-domains. Recently,  the first author and his co-author showed in \cite{HL21}  several different properties of plurifinely versus Euclidean topology.

This paper aims to establish the existing solutions of the complex Monge-Amp\`ere  equations in  $\mathcal F$-hyperconvex domain of $\mathbb C^n.$ %With some notions given in Section 2, we prove the following theorems.
Firstly, we recall the definition of the $\mathcal F$-hyperconvex domain from \cite{TVH}.

\begin{definition}
{\rm A bounded, connected, $\mathcal{F}$-open set $\Omega$  is called  $\mathcal F$-hyperconvex if  there exist a negative bounded plurisubharmonic function $\gamma_{\Omega} $ defined on a bounded hyperconvex domain $\Omega'\supset\Omega$  such that $\Omega =\{\gamma_{\Omega}  >-1\}$ and $-\gamma_{\Omega} $ is $\mathcal F$-plurisubharmonic  in $\Omega$. 
}\end{definition}

Observe that every bounded hyperconvex domain is $\mathcal F$-hyperconvex. Moreover, the first author and his co-author gave in  \cite{TVH} an example to show that a bounded $\mathcal F$-hyperconvex domain with no Euclidean interior point exists. 
Our first main theorem is the following theorem about the relationship of Cegrell's classes in the hyperconvex domain and the  $\mathcal F$-hyperconvex one.

 \begin{theorem}
\label{th21}
Let $\Omega $ be a bounded $\mathcal F$-hyperconvex domain in $\mathbb C^n$ and let $D $ be a bounded hyperconvex domain containing $\Omega$. If  $u\in \mathcal E(D)$ then $u|_\Omega \in \mathcal E(\Omega)$. 
\end{theorem}

The above result shows that Cegrell's classes can be naturally extended to bounded $\mathcal F$-hyperconvex domains. Therefore, we can surmise that the complex Monge-Amp\`ere equation is also solvable on pluripolar sets for the $\mathcal F$-plurisubharmonic functions. 
Moreover,  the pluripolar part of the complex Monge-Amp{\`e}re measures of the $\mathcal F$-plurisubharmonic functions is defined as follows (see  Theorem 1.3 and 1.4 in  \cite{CL22}).
\begin{definition}
Let $ D$ and $\Omega$ be as in Theorem \ref{th21}. Assume that   $u\in \mathcal F(\Omega) $ and define 
$$ \hat{u}=\hat{u}_D:=\sup \{ \varphi\in \mathcal F\text{-PSH}^-(D)   : \ \varphi \leq u  \text{ on } \Omega \}. $$
The pluripolar part $P(dd^cu)^n$ is defined as follows
$$ P(dd^cu)^n:=1_{D\cap \{\hat u =-\infty\}}(dd^c \hat u)^n .$$
\end{definition}

 The complex Monge-Amp\`ere equations for $\mathcal F$-plurisubharmonic functions is the problem of  finding a  function $u \in \mathcal F(\Omega) $ satisfying:  
$$
\mathcal {MA}(\Omega,\mu,\nu): 
\begin{cases}
P(dd^c u)^n=\mu   &\text{ in } \mathbb C^n,
		\\ NP(dd^c u)^n =\nu  &\text{ on } QB(\Omega).
\end{cases}
$$ 
Here, $\mu$ is a Borel measure in $\mathbb C^n$ and  $\nu$ is a non-negative measure  on $QB(\Omega)$. Moreover, a function $w\in \mathcal F(\Omega)$ is called sub-solution to $\mathcal {MA}(\Omega,\mu,\nu)$ if 
	$$\begin{cases} 
		 P(dd^c w)^n  \geq \mu  &\text{ in } \mathbb C^n ,
		\\  NP(dd^c w)^n  \geq \nu &\text{ on } QB(\Omega) .
			\end{cases}
			$$

When $\Omega$  is a bounded hyperconvex domain, U. Cegrell \cite{Ce2} proved the existence of classical plurisubharmonic solutions of  the problem in the case $\mu=0$ and $\nu(\Omega)<+\infty$. In 2009, P. \AA hag, U. Cegrell, R. Czy\. z and P. H. Hiep \cite{ACCH} showed that the problem can be solved if it has a sub-solution. Later, some authors investigated the existence of the problem (see \cite{Ce08, HHHP, HTH, Hong15b}).

When $\mu=0$, the first  author   gave in \cite{Hong17} sufficient conditions for which the problem   can be solved. Later on, the first and the second authors showed  in \cite{HC18b} that the problem  has solutions to a class of measures $\nu$ (also see \cite{HL21, HL22}). 

Our second main result is a result about the solvability of the problem. It is not surprising that we need to add the geometry property of $\Omega$. Specifically, we require it has the $\mathcal F$-approximation property. Here, a bounded $\mathcal F$-hyperconvex domain $\Omega$ in $\mathbb C^n$ has the $\mathcal F$-approximation property if there exists
a sequence of bounded hyperconvex domains $\Omega_j$   such that $\Omega \subset  \Omega_{j+1}\subset \Omega_j$  and  an increasing sequence of functions  $\psi_j\in \mathcal E_0(\Omega_j)$ that converges a.e. to a function $\psi \in \mathcal E_0(\Omega)$ on $\Omega$ (see \cite{Be06, HC18a, TVH}). We prove the following.

\begin{theorem} \label {th41}
Let $\Omega$ be a bounded $\mathcal F$-hyperconvex domain in $\mathbb C^n$ that has the $\mathcal F$-approximation property. 
%Assume that $\mu$ is a Borel measure  in $\mathbb C^n$ and   $\nu$ is a non-negative measure  on $QB(\Omega)$  such that $$\begin{cases}\mu \leq P(dd^c w)^n &\text{ in } \mathbb C^n ,\\ \nu \leq NP(dd^c w)^n &\text{ in } QB(\Omega) \end{cases}$$for some $w\in \mathcal F(\Omega)$.
If  the problem $\mathcal {MA}(\Omega,\mu,\nu)$ has a sub-solution $w$ then it can be solved.
\end{theorem}

Note that we haven't controlled the pluripolar part of the complex Monge-Amp{\`e}re measures of the $\mathcal F$-plurisubharmonic functions before. The above Theorem  also brings us more information about it.
Now, let $\Omega$ be a bounded $\mathcal F$-hyperconvex domain without Euclidean interior points. Such domains exist. Assume that $a\in \Omega$ and $\varphi\in \mathcal E_0(\Omega)$ such that $\varphi(a)<-1$. Let $r>0$ be such that 
$$\Omega\Subset B(a,r):=\{z\in \mathbb C^n: \|z-a\|<r\} .$$
Since $f(z):=\log\|z-a\| -\log r\in \mathcal F(B(a,r))$, Theorem \ref{th21} tells us that $f\in \mathcal E(\Omega)$. By Definition \ref{dn2.2}, we can find $ w\in \mathcal F(\Omega)$ such that  $w\geq f$ in $\Omega$ and
$$
w=f \text{ on } \Omega \cap \{\varphi<-1\}.
$$
Because $\hat w_{ B(a,r)}=f$ in $  B(a,r)$, we have
$$P(dd^c w)^n = 1_{\{f=-\infty\}}(dd^c f)^n .$$
Hence, $w$ will satisfy all the assumptions of Theorem \ref{th41} with
$$
\mu:=1_{\{f=-\infty\}}(dd^c f)^n \neq 0 \text{ and } \nu:=0.
$$
Thus, Theorem \ref{th41} is a generalization of the results of  P. \AA hag, U. Cegrell, R. Czy\. z and P. H. Hiep results in \cite{ACCH}.

Our paper is organized as follows. In Section 2, we study  Cegrell's classes of the $\mathcal F$-plurisubharmonic functions and prove Theorem \ref{th21}.  
Section 3 is devoted to the solvability of the  Monge-Amp\`ere equations in the class $\mathcal F(\Omega)$.

\section{Cegrell's classes of $\mathcal F$-plurisubharmonic functions}%Preliminaries}
%Some elements of pluripotential theory (plurifinely potential theory) that will be usedthroughout this article can be found in \cite{ACCH}-\cite{Xing}.  
Firstly, we recall the definition of the non-polar part of  $\mathcal F$-plurisubharmonic functions from \cite{KW14}.

\begin{definition}{\rm 
Let $\Omega\subset \mathbb C^n$ be an $\mathcal F$-open set   and let $\mathcal {F}\text{-}PSH(\Omega)$ be the $\mathcal F\text{-}$plurisubharmonic functions in $\Omega$. Denote by    $QB(\mathbb C^n)$  the measurable space on $\mathbb C^n$ generated by the Borel sets and the pluripolar subsets of  $\mathbb C^n$ and $QB(\Omega)$  is  the trace of $QB(\mathbb C^n)$ on $\Omega$. Assume that $u  \in \mathcal F\text{-} PSH(\Omega)$. Then, we can find  a pluripolar set $E $ and   bounded plurisubharmonic functions $f_{j}, g_{j}$ defined in Euclidean neighborhoods of $\mathcal F$-open sets $O_j$  such that   
 $$\Omega  =E \cup \bigcup_{j =1}^\infty O_j  \text{ and  } u=f_{j} -g_{j} \text{ on }O_j .$$  
 The  non-polar part $NP(dd^c u)^n$ of $\mathcal F$-plurisubharmonic function  $u$ is defined by
$$
\int_A  NP (dd^c u)^n     := \sum_{j=1}^\infty \int_{A\cap ( O_j \backslash \bigcup_{k=1}^{j-1} O_k ) }  (dd^c ( f_{j} -g_{j}))^n, \ \ A\in QB(\Omega).
$$  
}\end{definition}

The following definition of Cegrell's classes for $\mathcal F$-plurisubharmonic functions was given in \cite{TVH} (also in \cite{ACCH},\cite{Ce2}, \cite{Ce08},\cite{Xing}).
\begin{definition}
\label{dn2.2}
{\rm 
Let $\Omega$ be a bounded $\mathcal F$-hyperconvex domain in $\mathbb C^n $ and let $\gamma_\Omega$ be a negative bounded plurisubharmonic function  defined in a bounded hyperconvex domain $\Omega'$ such that  $-\gamma_\Omega $ is $\mathcal F$-plurisubharmonic in $\Omega$ and 
$$\Omega = \Omega' \cap \{\gamma_\Omega > -1 \}.$$

(a)
We say that a bounded, negative $\mathcal F$-plurisubharmonic function $ u $ defined on $\Omega$ belongs to $\mathcal E_0(\Omega)$ if 
$$ \int_\Omega  (dd^c u )^n < +\infty$$
 and for every $\epsilon >0$ there exists $\delta >0$ such that  
 $$\overline{\Omega \cap \{ u<-\epsilon \}} \subset \Omega' \cap\{ \gamma_\Omega >-1+\delta \}.$$

(b) Denote by $\mathcal F(\Omega)$  the family of $\mathcal F$-plurisubharmonic functions $u$ in $\Omega$  such that   there exists a decreasing sequence $\{u_j\}\subset \mathcal E_0(\Omega)$ that converges pointwise to $u$ on $\Omega  $ and satisfies 
$$
\sup_{j\geq 1} \int_\Omega  (dd^c u_j )^n < +\infty.
$$

(c) Let $\mathcal E(\Omega)$ be the set of $\mathcal F$-plurisubharmonic functions $u$ in $\Omega$ such that for each $\varphi\in \mathcal E_0(\Omega)$, there exists a function $v\in  \mathcal F(\Omega)$ satisfying $v\geq u$ in $\Omega$ and $v=u$ in $\Omega\cap \{ \varphi<-1\}$.
} \end{definition}

\begin{lemma}
\label{le30}
	Let $\Omega$ be a bounded $\mathcal F$-hyperconvex domain in $\mathbb C^n$ and let $\varphi\in\mathcal E_0(\Omega)$. Then, for every $\varepsilon>0$, there exist bounded plurisubharmonic functions $\tilde \varphi$ and $\tilde \varphi_\varepsilon$ defined in bounded hyperconvex domain $\Omega' \supset \Omega$ such that 
\begin{align*}
	\Omega \cap  \{\varphi<-\varepsilon\} =  \Omega' \cap \{\tilde \varphi<\tilde \varphi_\varepsilon\}.
\end{align*} 
\end{lemma}

\begin{proof}
Without loss of generality we can assume that $-1<\varphi<0$ in $\Omega$. We set 
$$\varphi_\varepsilon:=\max(\varphi, -\varepsilon).$$ 
Let  $\gamma$ be a negative bounded plurisubharmonic function defined on a bounded hyperconvex domain $\Omega'$  such that   $-\gamma$ is $\mathcal F$-plurisubharmonic in $\Omega$ and satisfies 
	$$
		\Omega= \Omega' \cap \{\gamma  >-1\}.
	$$  
Let  $\delta\in(0,1)$ be such that 
\begin{equation}
	\label{e31-}
	\Omega\cap \{ \varphi \leq -\varepsilon\} \subset \Omega'\cap \{\gamma > -1+ 2 \delta\}.
\end{equation}
Assume that $f\in \mathcal F \text{-}PSH(\Omega)$ with $-1<f<0$ and define
$$\tilde f := 
\begin{cases}
\max(-  \delta^{-1}, f +\delta^{-1} \gamma  ) & \text{ in } \Omega,
\\ -  \delta^{-1}& \text{ in } \Omega' \backslash \Omega .
\end{cases}
$$  
Proposition 2.3 in \cite{KS14} tells us that $\tilde f$ is a $\mathcal F$-plurisubharmonic function in $\Omega'$, and hence, Proposition 2.14 in \cite{KFW11} implies that  $\tilde f$ is plurisubharmonic in $\Omega'$ because $\Omega'$ is a  Euclidean open set. 
It is easy to see that. 
$$
f= \tilde f - \delta^{-1} \gamma\text{ on } \Omega' \cap \{\gamma>-1 + \delta \}.
$$
 We deduce by \eqref{e31-} that 
\begin{align*}
	\Omega \cap  \{\varphi<-\varepsilon\} = \Omega \cap \{\varphi<\varphi_\varepsilon\}=  \Omega' \cap \{\tilde \varphi<\tilde \varphi_\varepsilon\}.
\end{align*} 
	The proof is complete. 
\end{proof}

\begin{lemma}
\label{le30'}
	Let $\Omega$ be a bounded $\mathcal F$-hyperconvex domain in $\mathbb C^n$ and let $\varphi \in \mathcal E_0(\Omega)$. Assume that  $u$  is  a  $\mathcal F$-plurisubharmonic function in $\Omega$. Then, $u \in  \mathcal E(\Omega)$ if and only if there exists a decreasing sequence $\{u_j\} \subset \mathcal F(\Omega)$ such that 
	$$
		u_j =u \text{ on } \Omega \cap \{j \varphi<-1\}.
	$$
\end{lemma}

\begin{proof}
Since $j\varphi \in \mathcal E_0(\Omega)$,   there exists a function $v_j \in \mathcal F(\Omega)$ such that 
	$$
		v_j =u \text{ on } \Omega \cap \{j \varphi<-1\}.
	$$
	Let $u_j$ be the $\mathcal F$-upper semi-continuous majorant  of $\sup_{k\geq j} v_k $	in $\Omega$.  
	It is easy to see that $\{u_j\}$ is a decreasing sequence of $\mathcal F$-plurisubharmonic functions in $\Omega$. 
	Since 
	 $$\{j\varphi<-1\} \subset \{k\varphi<-1\}, \ \forall k\geq j,$$
	 it follows that 
	 $$v_k=u \text{ on } \{j\varphi<-1\}, \ \forall k\geq j,$$
	 and therefore, 
	 	$$
	u_j:=u\text{ on } \Omega \cap \{j \varphi<-1\}.
	$$
	We now deduce by the definition of $u_j$ that. 
	$$
	v_j \leq u_j<0 \text{ in } \Omega,
	$$
	and hence, $u_j \in \mathcal F(\Omega)$. 
	This proves the lemma. 
\end{proof}

\begin{lemma} 
\label{le27}
	Let $\Omega$ be a bounded $\mathcal F$-hyperconvex domain in $\mathbb C^n$ and let $\{u_j\}\subset \mathcal F(\Omega)$ be  a decreasing sequence such that 
	$$
	\sup_{j\geq 1} \int_{\Omega} (dd^c \max(u_j,-1))^n<+\infty.
	$$
	Then,  $u:=\lim_{j\to+\infty} u_j \in \mathcal F(\Omega)$. 
\end{lemma}

\begin{proof}
	Let $\varphi \in \mathcal E_0(\Omega)$ such that $\varphi  < 0$ in $\Omega$  . Put $\varphi_j=\max\{ u_j,j\varphi \},$ then $\varphi_j \in \mathcal E_0(\Omega)$ for all $j\geq 1$.
Because of the fact that $u_j \searrow u$  as $j\to +\infty,$ $\varphi_j \searrow u$  as $j\to +\infty,$ too. Moreover, we infer from Proposition 4.2 in \cite{TVH} and Proposition 4.3 in \cite{TVH} that 
\begin{align*}
	\sup_{j\geq 1} \int_{\Omega} (dd^c \varphi_j)^n
	&=\sup_{j\geq 1} \int_{\Omega} (dd^c \max\{ u_j,j\varphi,-1 \})^n
	\\& \leq \sup_{j\geq 1} \int_{\Omega} (dd^c \max(u_j,-1))^n<+\infty .
\end{align*}
This implies that $u:=\lim_{j\to+\infty} u_j \in \mathcal F(\Omega)$. This proves the lemma.
\end{proof}

%\section{ The class $\mathcal E(\Omega)$}

\begin{lemma}\label{le1} 
Let $\Omega$ be a bounded $\mathcal F$-hyperconvex domain and let  $\gamma$ be
a negative bounded plurisubharmonic function defined on a bounded hyperconvex domain $\Omega'$  such that   $-\gamma$ is $\mathcal F$-plurisubharmonic in $\Omega$ and satisfies 
	$$
		\Omega= \Omega' \cap \{\gamma  >-1\}.
	$$  
Assume that $\varphi \in\mathcal E_0(\Omega)$ and $\varepsilon,\delta\in(0,1)$   such that 
\begin{equation}\label{e11}
	\Omega\cap \{ \varphi \leq - \frac{\varepsilon }{4}\} \subset \Omega'\cap \{\gamma > -1+ 2 \delta\}.
\end{equation}
Then, there exists   $ \tilde \varphi_1,  \tilde \varphi_2\in PSH(\Omega') \cap L^\infty(\Omega')$   such that    for every $p\geq 0$ and for every   $u,v, w_1, \ldots, w_{n-1} \in \mathcal F \text{-} PSH(\Omega) \cap L^\infty(\Omega)$ with  $u\leq v <0$ in $\Omega$, we have 
\begin{equation} \label{1}
\begin{split}
\frac{p}{2} \int\limits_{\{\varphi\leq -\varepsilon \}}(-u)^{p-1}  du\wedge d^cu \wedge T 
\leq & \frac{1}{p+1}\int\limits_{\{\varphi\leq -\frac{\varepsilon }{2}\}} ( -   u)^{p+1}  dd^c \tilde\varphi_1  \wedge  T \\& +  \int\limits_{\{\varphi\leq -\frac{\varepsilon }{2} \}}( -u)^p  dd^c u  \wedge  T,
 \end{split}
\end{equation}
and
\begin{equation} \label{2}
	\begin{split} 
	\int_{\{\varphi\leq -\varepsilon \}} (-u)^p  dd^cv\wedge T 
				     \leq    (p+2)   & \int_{ \{ \varphi \leq - \frac{\varepsilon }{4}\} }    (-u)^{p}   dd^c u   \wedge T
				   \\  +      \left( \frac{4p(1+\delta) e^{\frac{1}{\delta}} }{\delta}
				    + \frac{2e^{\frac{1}{\delta} }}{p+1} \right)  
				    &  \int_{ \{ \varphi \leq - \frac{\varepsilon }{4}\} }     (- u)^{p+1}   dd^c ( e^{\tilde \varphi_1} + e^{\tilde \varphi_2}) \wedge T.  
	\end{split}
\end{equation}
Here,  
$$T:= dd^c w_1\wedge \ldots \wedge dd^c w_{n-1}.$$
\end{lemma}

\begin{proof}  
Since $\frac{\varphi}{\varepsilon} \in \mathcal E_0(\Omega)$, by replacing $\varphi$ with $\frac{\varphi}{\varepsilon}$ if necessary, we can assume that
$
\varepsilon=1.
$
Moreover, without loss of generality we can assume that 
$$-1 \leq u,v, w_1, \ldots, w_{n-1}< 0 \text{ in } \Omega.$$ 
Now assume that $-1\leq f\leq 0 $ is  a $\mathcal F$-plurisubharmonic function in $\Omega$ and define 
$$\tilde f   := 
\begin{cases}
\max(-  \delta^{-1},  f  + \tilde \gamma  ) & \text{ in } \Omega,
\\ -  \delta^{-1}& \text{ in } \Omega' \backslash \Omega.
\end{cases}
$$  
Here, $\tilde \gamma:= \delta^{-1} \gamma $.
Proposition 2.3 in \cite{KS14} tells us that $\tilde f$ is a $\mathcal F$-plurisubharmonic function in $\Omega'$, and hence, Proposition 2.14 in \cite{KFW11} implies that  $\tilde f$ is plurisubharmonic in $\Omega'$ because $\Omega'$ is a  Euclidean open set. Moreover, \begin{equation}\label{e12}
f= \tilde f - \tilde \gamma \text{ on } \Omega' \cap \{\gamma>-1 + \delta \}.
\end{equation}
Put
$$\varphi_1:= \max(\varphi,-1) \text{ and }   \varphi_2:= \max(\varphi,-\frac{1}{2}).$$  
We deduce by  \eqref{e11} that 
$$
\Omega \cap \{\gamma <-1+ 2 \delta\} \subset \Omega \cap \{\varphi>-\frac{1}{2} \} .
$$
Hence, 
\begin{equation}\label{e13}
\tilde \varphi_1=\tilde \varphi_2  \text{ on } (\Omega'\cap \{\gamma <  -1+ 2 \delta\})\cup  (\Omega \cap \{ \varphi >- \frac{1}{2}\} ).
\end{equation}
It is easy to see that 
$\varphi_1\leq \varphi_2$ in $\Omega$ and 
\begin{equation}\label{e14}
\varphi_2-\varphi_1=\frac{1}{2} \text{ on } \Omega \cap \{\varphi\leq - 1\}.
\end{equation}
Set
$$
\tilde T:= dd^c (\tilde w_1-\tilde  \gamma ) \wedge \ldots \wedge dd^c (\tilde w_{n-1}- \tilde  \gamma) .
$$
We obtain from \eqref{e11},  \eqref{e12} and \eqref{e13} that 
\begin{align*}
& p(\varphi_2-\varphi_1)   (-u)^{p-1}  du\wedge d^cu \wedge T
\\ & =-  (\tilde \varphi_2- \tilde\varphi_1) d(\tilde u- \tilde  \gamma)^p\wedge d^c(\tilde u-\tilde  \gamma ) \wedge \tilde T  
\text{ on }\Omega.	
\end{align*}
Therefore, we infer by   \eqref{e13} and \eqref{e14}  that
\begin{equation}
\label{e15a}
\begin{split}
& \frac{p}{2} \int\limits_{\Omega\cap \{\varphi\leq -1\}}(-u)^{p-1}  du\wedge d^cu \wedge T
\\& \leq \int_{\Omega'}  -  (\tilde \varphi_2- \tilde \varphi_1) d(\tilde u- \tilde  \gamma)^p\wedge d^c(\tilde u-\tilde  \gamma ) \wedge \tilde T
\end{split}
\end{equation}
because 
$$
\tilde \varphi_2- \tilde \varphi_1 =0
\text{ on }\Omega'\backslash \Omega.$$
Integration by parts tells us that 
\begin{align*}
&\int_{\Omega'}  -  (\tilde \varphi_2- \tilde \varphi_1) d(\tilde u- \tilde  \gamma)^p\wedge d^c(\tilde  u-\tilde  \gamma ) \wedge \tilde T
\\&=  \int\limits_{\Omega'} (\tilde \gamma - \tilde u)^p d(\tilde \varphi_2- \tilde \varphi_1) \wedge d^c(\tilde u- \tilde \gamma) \wedge  \tilde T + \int\limits_{\Omega'} (\tilde \varphi_2- \tilde \varphi_1) (\tilde  \gamma - \tilde u)^p  dd^c (\tilde u- \tilde  \gamma) \wedge \tilde T 
\\&= \frac{-1}{p+1}\int\limits_{\Omega'} (\tilde \gamma - \tilde u)^{p+1}  dd^c (\tilde \varphi_2- \tilde \varphi_1) \wedge \tilde  T  +  \int\limits_{\Omega'} (\tilde \varphi_2- \tilde \varphi_1)   (\tilde \gamma - \tilde u)^p dd^c(\tilde u- \tilde  \gamma) \wedge   \tilde  T.\end{align*} 
Hence,
we deduce from  \eqref{e13} and \eqref{e15a}  that 
\begin{align*}
&\frac{p}{2} \int\limits_{\{\varphi\leq -1\}}(-u)^{p-1}  du\wedge d^cu \wedge T
\\& \leq \frac{-1}{p+1}\int\limits_{\{\varphi\leq -\frac{1}{2}\}} ( -   u)^{p+1}  dd^c ( \varphi_2- \varphi_1) \wedge  T  +  \int\limits_{\{\varphi\leq -\frac{1}{2} \}}  ( \varphi_2-   \varphi_1)   (  -  u)^p  dd^c u  \wedge  T
\\&\leq  \frac{1}{p+1}\int\limits_{\{\varphi\leq -\frac{1}{2}\}} ( -   u)^{p+1}  dd^c   \varphi_1  \wedge  T  +  \int\limits_{\{\varphi\leq -\frac{1}{2} \}}( -u)^p  dd^c u  \wedge  T. 
\end{align*} 
This proves \eqref{1}. We now give the proof of \eqref{2}. Observe that
$$
(\varphi_2-\varphi_1)   (-u)^{p}  d d^c v\wedge T
=  (\tilde\varphi_2- \tilde \varphi_1) (\tilde u- \tilde \gamma)^p d  d^c(\tilde v-\tilde  \gamma ) \wedge \tilde T  
\text{ on }\Omega.
$$ 
We deduce from  \eqref{e13} that 
		\begin{align*} 		
		 dd^c((\tilde \varphi_2- \tilde \varphi_1) ( \tilde  \gamma -\tilde u)^p ) \wedge \tilde T =0 \text{ on } (\Omega'\cap \{\gamma <  -1+   2\delta\})\cup  (\Omega \cap \{ \varphi >- \frac{1}{2}\} ).
	\end{align*}
Using integration by parts we have
\begin{equation}\label{e15}	
		\begin{split} 		
		\frac{1}{2}\int_{\{\varphi\leq -1\}} |u|^p  dd^cv\wedge T
		& \leq   \int_{\Omega'}(\tilde \varphi_2- \tilde \varphi_1) ( \tilde \gamma -\tilde u)^p d  d^c(\tilde v-\tilde  \gamma ) \wedge \tilde T
		\\&=  \int_{\Omega'} (\tilde v-\tilde  \gamma ) dd^c((\tilde \varphi_2- \tilde \varphi_1) ( \tilde  \gamma -\tilde u)^p ) \wedge \tilde T 
		\\&= \int_{ \Omega \cap \{ \varphi \leq - \frac{1}{2}\} } (\tilde v-\tilde  \gamma ) dd^c((\tilde \varphi_2- \tilde\varphi_1) ( \tilde  \gamma -\tilde u)^p ) \wedge \tilde T .
	\end{split}
\end{equation} 
By computation we have
\begin{equation}\label{e16}	
	\begin{split}
		& dd^c((\tilde \varphi_2- \tilde \varphi_1) ( \tilde \gamma -\tilde u)^p )  
		\\ &=    ( \tilde \gamma -\tilde u)^pdd^c ( \tilde \varphi_2- \tilde \varphi_1)+   ( \tilde \varphi_2- \tilde \varphi_1) dd^c  (\tilde \gamma -\tilde u)^p + d ( \tilde \varphi_2- \tilde \varphi_1) \wedge d^c ( \tilde  \gamma -\tilde u)^p
		\\& \ \ \ \ \ \  + d  ( \tilde  \gamma -\tilde u)^p \wedge d^c ( \tilde \varphi_2- \tilde \varphi_1)  
				     \\  &\geq   -  ( \tilde  \gamma -\tilde u)^p dd^c   \tilde \varphi_1     - p  (\tilde \varphi_2- \tilde \varphi_1) ( \tilde  \gamma -\tilde u)^{p-1}   dd^c( \tilde u - \tilde  \gamma)    
				       \\ & \ \ \ \      -p( \tilde  \gamma -\tilde u)^{p-1} \left(d ( \tilde \varphi_2- \tilde\varphi_1) \wedge d^c( \tilde u - \tilde  \gamma) + d( \tilde u - \tilde  \gamma)  \wedge d^c ( \tilde \varphi_2- \tilde \varphi_1)\right) 
	\end{split}
\end{equation} 
on $ \Omega' \cap \{\gamma>-1 + \delta \}$. 
We infer by \eqref{e13} that 
\begin{equation}\label{e17}	
( \tilde \varphi_2- \tilde \varphi_1)  ( \tilde  \gamma -\tilde u)^{p-2}d ( \tilde u - \tilde  \gamma)   \wedge d^c( \tilde u - \tilde  \gamma)  \geq 0 \text{ on } \Omega'.
\end{equation} 
Using the basic inequality
\begin{equation}
	\label{e18}
	\pm (f+a)^k(df\wedge d^cg + dg\wedge d^cf)\geq  -|f+a|^{k-1} df\wedge d^cf - |f+a|^{k+1} d g\wedge d^c g, \ \forall a\in \mathbb R,
\end{equation}
	we get 
	\begin{equation}
	\label{e19}
\begin{split}
	&  ( \tilde  \gamma -\tilde u)^{p-1} \left[d (\tilde \varphi_2-\tilde \varphi_1) \wedge d^c( \tilde u - \tilde  \gamma) + d( \tilde u - \tilde  \gamma)  \wedge d^c ( \tilde \varphi_2- \tilde \varphi_1)\right]
	\\&  \geq   - |\tilde  \gamma -\tilde u|^{p-2} d ( \tilde u - \tilde  \gamma)  \wedge d^c( \tilde u - \tilde  \gamma) -  |\tilde  \gamma -\tilde u|^{p}    d( \tilde \varphi_2- \tilde \varphi_1)   \wedge d^c ( \tilde \varphi_2- \tilde \varphi_1)  .
		\end{split}
		\end{equation}
Since $-\frac{1}{\delta}\leq \tilde  \varphi_1\leq  \tilde \varphi_2\leq 0$ in $\Omega'$, again using \eqref{e18} we obtain that 
\begin{align*}
	&  -    d( \tilde \varphi_2- \tilde \varphi_1)   \wedge d^c ( \tilde \varphi_2 - \tilde\varphi_1)  
		\\& = 		-   	 	d \tilde \varphi_1   \wedge d^c \tilde  \varphi_1  - d\tilde \varphi_2   \wedge d^c  \tilde \varphi_2  
		+    d \tilde  \varphi_1  \wedge d^c  \tilde \varphi_2 +d \tilde \varphi_2   \wedge d^c \tilde \varphi_1 	
		\\& \geq  		-     d \tilde  \varphi_1   \wedge d^c \tilde  \varphi_1  - d \tilde\varphi_2   \wedge d^c  \tilde \varphi_2  
		  -     \frac{1}{|\tilde \varphi_1+\frac{2}{\delta}|} d \tilde \varphi_1   \wedge d^c  \tilde \varphi_1 - |\tilde\varphi_1+\frac{2}{\delta}|d \tilde \varphi_2   \wedge d^c  \tilde \varphi_2  	
		\\& \geq -  \frac{2(1+\delta) e^{\frac{1}{\delta}} }{\delta}   dd^c (e^{\tilde \varphi_1} + e^{\tilde \varphi_2}) .
		\end{align*}
		Combining this with   \eqref{e16}, \eqref{e17} and \eqref{e19} we arrive that 
	\begin{equation} \label{e20}
		\begin{split}
		&dd^c( (\tilde \varphi_2-\tilde \varphi_1)  ( \tilde  \gamma -\tilde u)^p )  
				     \\ &\geq   -  ( \tilde  \gamma -\tilde u)^p dd^c    \tilde \varphi_1     - p  ( \tilde \varphi_2- \tilde \varphi_1) ( \tilde  \gamma -\tilde u)^{p-1}   dd^c( \tilde u - \tilde  \gamma)    
				       \\ & \ \ \ \      - p (\tilde  \gamma -\tilde u)^{p-2} d ( \tilde u - \tilde  \gamma)  \wedge d^c( \tilde u - \tilde  \gamma)
				         -    \frac{2p(1+\delta) e^{\frac{1}{\delta}} }{\delta} (\tilde  \gamma -\tilde u)^{p}   dd^c (e^{\tilde \varphi_1} + e^{\tilde  \varphi_2})
				    \\&\geq     - p  ( \tilde \varphi_2- \tilde \varphi_1) ( \tilde  \gamma -\tilde u)^{p-1}   dd^c( \tilde u - \tilde  \gamma)      - p (\tilde  \gamma -\tilde u)^{p-2} d ( \tilde u - \tilde  \gamma)  \wedge d^c( \tilde u - \tilde  \gamma)
				       \\&   \ \ \ \      -    \frac{4p(1+\delta) e^{\frac{1}{\delta}} }{\delta} (\tilde  \gamma -\tilde u)^{p}   dd^c (e^{\tilde\varphi_1} + e^{\tilde \varphi_2})
				       \\&=    - p   (-u)^{p-1}   dd^c u    - p ( -u)^{p-2} d u  \wedge d^c u -    \frac{4p(1+\delta) e^{\frac{1}{\delta}} }{\delta} (- u)^{p}   dd^c (e^{\tilde \varphi_1} + e^{\tilde \varphi_2})
	\end{split}
	\end{equation}
on $ \Omega' \cap \{\gamma>-1 + \delta \}$.  
On the other hand,  we obtain from \eqref{1}  that 
	\begin{equation} \label{e21}
\begin{split}
	p&\int\limits_{\{\varphi\leq -1/2\}}(-u)^{p-1}  du\wedge d^cu \wedge T
\\ & \leq  \frac{2  }{p+1}\int\limits_{\{\varphi\leq -\frac{1}{4}\}} ( -   u)^{p+1}  dd^c   \tilde  \varphi_1   \wedge  T  +  2\int\limits_{\{\varphi\leq -\frac{1}{4} \}}( -u)^p  dd^c u  \wedge  T.
\\ & \leq  \frac{2e^{\frac{1}{\delta} }}{p+1}\int\limits_{\{\varphi\leq -\frac{1}{4}\}} ( -   u)^{p+1}  dd^c  e^{\tilde \varphi_1}  \wedge  T  +  2\int\limits_{\{\varphi\leq -\frac{1}{4} \}}( -u)^p  dd^c u  \wedge  T. 
\end{split}
	\end{equation}
Since $ \{ \varphi \leq - \frac{1}{2}\} \subset \{ \varphi \leq - \frac{1}{4}\} \subset \Omega' \cap \{\gamma>-1 + \delta \}$ and $u\leq v<0$ on $\Omega$, we deduce by \eqref{e15}, \eqref{e20} and \eqref{e21} that 
\begin{align*}
	&\int_{\{\varphi\leq -1\}} (-u)^p  dd^cv\wedge T 
				   \\& \leq   \int_{ \{ \varphi \leq - \frac{1}{2}\} }  v \left[ - p   (-u)^{p-1}   dd^c u    - p ( -u)^{p-2} d u  \wedge d^c u -    \frac{4p(1+\delta) e^{\frac{1}{\delta}} }{\delta} (- u)^{p}   dd^c (e^{\tilde \varphi_1} + e^{\tilde \varphi_2}) \right] \wedge T
\\& \leq   \int_{ \{ \varphi \leq - \frac{1}{2}\} }    \left[  p   (-u)^{p}   dd^c u    + p ( -u)^{p-1} d u  \wedge d^c u+    \frac{4p(1+\delta) e^{\frac{1}{\delta}} }{\delta} (- u)^{p+1}   dd^c (e^{\tilde \varphi_1} + e^{\tilde \varphi_2}) \right] \wedge T
				   \\&  \leq   \int_{ \{ \varphi \leq - \frac{1}{4}\} }    \left[  (p+2)   (-u)^{p}   dd^c u    +   \left( \frac{4p(1+\delta) e^{\frac{1}{\delta}} }{\delta}+ \frac{2e^{\frac{1}{\delta} }}{p+1} \right) (- u)^{p+1}   dd^c (  e^{\tilde \varphi_1} + e^{\tilde \varphi_2}) \right] \wedge T,
\end{align*}
which completes the proof.
\end{proof}

We now able to give the proof of Theorem \ref{th21}.

\begin{proof}[Proof of Theorem \ref{th21}]
Let  $\gamma$ be a negative bounded plurisubharmonic function  defined on a bounded hyperconvex domain $\Omega' $  such that   $-\gamma$ is $\mathcal F$-plurisubharmonic in $\Omega$ and satisfies 
	$$
		\Omega= \Omega' \cap \{\gamma  >-1\}.
	$$  
	By replacing $\Omega'$ with $D\cap\Omega'$, we can assume that $\Omega' \subset D$, and hence, $u\in \mathcal E(\Omega')$. 
	Let $\varphi\in \mathcal E_0(\Omega)$ and let $G$ be an open set such that 
	$$
	\overline{\Omega \cap \{\varphi\leq -\frac{1}{4^n}\}} \Subset G\Subset \Omega'.
	$$ 
	Since $u\in \mathcal E(\Omega')$, there exists a decreasing sequence of functions $\{u_j \} \subset \mathcal E_0(\Omega')$ such that $  u_j \to u$ on $G$ and 
	$$
	\sup_{j\geq 1} \int_{\Omega'} (dd^c u_j)^n<+\infty.
	$$
	Define 
	$$
	v_j:= \sup \{\psi \in \mathcal F \text{-} PSH^-(\Omega): \psi\leq u_j \text{ on } \Omega \cap  \{\varphi< -1\}\}.
	$$
	Since $\Omega \cap  \{\varphi< -1\}$ is $\mathcal F$-open, so $v_j$ is $\mathcal F$-upper semi-continuous on $\Omega$ and	
	$$\|u_j\|_{L^\infty(D)} \varphi  \leq v_j  \text{ on } \Omega .$$
	It follows that  $v_j \in \mathcal E_0(\Omega)$, and hence, $\{v_j\} \subset \mathcal E_0(\Omega)$ is decreasing. 
	Proposition 3.2 in \cite{KS14} tells us that $v_j$ is $\mathcal F$-maximal on $\Omega \cap \{\varphi>-1\}$, and hence, Theorem 4.8 in \cite{KS14} implies that 
\begin{equation}
	\label{e23}
	(dd^c v_j)^n =0 \text{ on } \Omega \cap \{\varphi>-1\}.
\end{equation}
	 Since $u_j\leq v_j<0$ in $\Omega$,  Lemma \ref{le1} tells us that for every $p,q\geq 0$, there exist a positive constant $c_{p,q}$ and   bounded plurisubharmonic functions   $\psi_{p,q}, \varphi_{p,q}$ in $\Omega'$ such that 
	 \begin{equation}  
	 \label{e22}
	\begin{split} 
	\int_{\{\varphi\leq - \frac{1}{4^{q-1}} \}} &(-u_j)^p  dd^cv_j\wedge  dd^cw_1\wedge  \ldots \wedge  dd^c w_{n-1} 
	\\ & \leq   c _{p,q}   \int_{ \{ \varphi  \leq - \frac{1}{4^{q }}\} }     (-u_j)^{p}   dd^c u_j   \wedge   dd^cw_1\wedge  \ldots \wedge  dd^c w_{n-1} 
	\\ &+   c _{p,q} \int_{ \{ \varphi  \leq - \frac{1}{4^{q }}\} }       (- u_j)^{p+1}   dd^c (e^{\psi _{p,q}} +e^{\varphi _{p,q}})   \wedge   dd^cw_1\wedge  \ldots \wedge  dd^c w_{n-1}  
	\end{split}
	\end{equation}
	for all $j\geq 1$ and for all bounded $\mathcal F$-plurisubharmonic functions $w_1,\ldots,w_{n-1}$ in $\Omega$. 	Let $\psi \in \mathcal E_0(\Omega')$ be such that  
	$$
	\psi := \sum_{p,q=0}^n( e^{\psi _{p,q}} +e^{\varphi _{p,q}})  \text{ on } G. 
	$$
	Set
	$$
	c_1:=   \sum_{p,q=0}^n c_{p,q}
	$$
%	Then, $\psi$ is a bounded plurisubharmonic function in $\Omega'$. 
From \eqref{e22} we arrive that  
		 \begin{equation}  
	 \label{e24}
	\begin{split}  
	&\int_{\{\varphi\leq - \frac{1}{4^{s}} \}}  (-u_j)^p (dd^c u_j)^{s-p}  \wedge (dd^c \psi )^p \wedge   (dd^c v_j)^{n -s}
	\\ & \leq    c _1  \sum_{q=0}^1  \int_{ \{ \varphi  \leq - \frac{1}{4^{s+1 }}\} }       (- u_j)^{p+q}   (dd^c u_j)^{s+1-p-q} \wedge (dd^c \psi )^{p+q} \wedge    (dd^c v_j)^{n-s-1} 
	\end{split}
	\end{equation}
	for all $j\geq 1$ and for all $0\leq p\leq s\leq n-1$.		By applying \eqref{e24} many times, we infer that 
\begin{equation*}
	\label{e25} 
	\begin{split}
		\int_{\{\varphi\leq -1\} }    (dd^cv_j)^{n } 
		& \leq    c _1  \sum_{p=0}^1  \int_{ \{ \varphi  \leq - \frac{1}{4}\} }   (- u_j)^{p}   (dd^c u_j)^{1-p} \wedge (dd^c \psi )^p \wedge   (dd^c v_j)^{n-1}  
				\\ & \leq    c _1^2   \sum_{p=0}^1    \sum_{q=0}^1  \int_{ \{ \varphi  \leq - \frac{1}{4^2}\} }   (- u_j)^{p+q}  (dd^c u_j)^{2-p -q } \wedge (dd^c \psi )^{p+q} \wedge   (dd^c v_j)^{n-2}   
								\\ & \leq    c_2 \sum_{p=0}^2  \int_{ \{ \varphi  \leq - \frac{1}{4^2}\} }   (- u_j)^{p}   (dd^c u_j)^{2-p}  \wedge (dd^c \psi )^p \wedge   (dd^c v_j)^{n-2}   
        \\&\leq \ldots   \leq    c_{n} \sum_{p=0}^n  \int_{ \{ \varphi  \leq - \frac{1}{4^n}\} }   (- u_j)^{p}   (dd^c u_j)^{n-p}  \wedge (dd^c \psi )^p  . 
\end{split}
\end{equation*}
Combining this with \eqref{e23} we obtain that 
\begin{equation}
	\label{e26}
	\int_{\Omega } (dd^cv_j)^n  
\leq     c_{n} \sum_{p=0}^n  \int_{ \Omega ' }   (- u_j)^{p}   (dd^c u_j)^{n-p}  \wedge (dd^c \psi )^p.
\end{equation}
Now, using integration by parts  we have 
	\begin{align*}
		     I_p&:= \int_{ \Omega' }   (- u_j)^{p}   (dd^c u_j)^{n-p}  \wedge (dd^c  \psi )^p  		     
		     \\&=
		   \int_{\Omega'} \psi  (dd^cu_j)^{n-p}\wedge (dd^c \psi )^{p-1}\wedge dd^c  (- u_j)^{p} 
		   \\
		&=   \int_{\Omega'} \psi  (dd^cu_j)^{n-p}\wedge (dd^c \psi  )^{p-1}\wedge [p(p-1) (-u_j)^{p-2} du_j\wedge d^cu_j -p(-u_j)^{p-1} dd^cu_j]\\
		&\leq   p \int_{\Omega'} (-\psi ) (-u_j)^{p-1} (dd^cu_j)^{n-p+1}\wedge (dd^c \psi  )^{p-1}
		\\&\leq   p \|\psi\|_{L^\infty(\Omega')} \int_{\Omega'}  (-u_j)^{p-1} (dd^cu_j)^{n-p+1}\wedge (dd^c \psi  )^{p-1}
		\\&=p \|\psi\|_{L^\infty(\Omega')} I_{p-1} .
       \end{align*}
This implies that 
$$
I_p\leq p! (\|\psi\|_{L^\infty(\Omega')} )^p \int_{ \Omega' } (dd^c u_j)^{n}.
$$ 
Hence, we deduce by \eqref{e26} that 
$$
\sup_{j\geq 1} \int_{ \Omega } (dd^c v_j)^{n}
 \leq  c_n  \left( \sum_{p=0}^n p! (\|\psi\|_{L^\infty(\Omega')} )^p
 \right)\sup_{j\geq 1} \int_{ \Omega' } (dd^c u_j)^{n}<+\infty. 
$$
Therefore, $$v:=\lim_{j\to+\infty} v_j \in \mathcal F(\Omega).$$
Since $u=v$ on $\Omega\cap \{\varphi<-1\}$, we conclude that $u\in \mathcal E(\Omega)$. The proof is complete. 
\end{proof}

\section{Complex Monge-Amp\`ere equations}

In this section, we give proof of our main result, Theorem \ref{th41}. Firstly, we prove the solvability of the problem stated in Theorem \ref{th41} in the case $\nu=0.$ Then, we solve this problem with $\nu$ arbitrary.

First of all, we need the following lemma.  

\begin{lemma} 
\label{le28}
	Let $\Omega$ be a bounded $\mathcal F$-hyperconvex domain in $\mathbb C^n$ that has the $\mathcal F$-approximation property. Assume that $u\in \mathcal F\text{-} PSH^-(\Omega)$ and $v\in\mathcal F( \Omega)$  such that 
	\begin{equation}
		\label{e221}
		w +v\leq u \leq v \text{ on } \Omega, \text{ for some }  w\in\mathcal F^a( \Omega).
	\end{equation}
	Then, $u\in\mathcal F( \Omega)$ and 
	$$
	P(dd^c u)^n =P (dd^c v)^n \text{ in } \mathbb C^n.
	$$
\end{lemma}
\begin{proof}
Since   $\Omega$ has  the $\mathcal F$-approximation property,  we can find  a  decreasing  sequence of    bounded  hyperconvex domains $\{\Omega_j\}$ and a  sequence of  functions  $\psi_j \in \mathcal E_0(\Omega_j)$ such that    $\Omega\subset\Omega_{j+1}\subset\Omega_j$ and 
$$ \psi_j \nearrow  \psi \in \mathcal E_0(\Omega) \text{  a.e.  on }\Omega.$$  
	For $f\in \mathcal F \text{-} PSH^-(\Omega)$, we use the symbol
	$$\hat f_j:=\sup \{ \varphi\in \mathcal F\text{-PSH} (\Omega_j)  : \ \varphi \leq f  \text{ on } \Omega \}.$$ 
	Theorem 1.2 in \cite{HC18a} tells us that 
	$$
	  \hat w_j \nearrow w \text{ and } \hat v_j   \nearrow v \text{ a.e. on } \Omega.
	$$
	Observed that
	\begin{equation}
		\label{e222}
	\mathcal F(\Omega_j) \ni \max( \hat w_j +\hat v_j, k \psi_j) \nearrow   \max( w  + v , k \psi ) \in \mathcal E_0(\Omega) \  \text{ a.e. in } \Omega, \ \forall k\geq 1.
	\end{equation}  
	Since $\hat v_j$ and $ \hat w_j $ belong to the class $\mathcal F(\Omega_j)$ so 
	$$
	 \left[ \int_{\Omega_j}  (dd^c  ( \hat w_j +\hat v_j ))^n \right]^{1/n} 
		\leq  \left[ \int_{\Omega_j}  (dd^c    \hat w_j  )^n  \right]^{1/n}
		+  \left[  \int_{\Omega_j}  (dd^c  \hat v_j  )^n \right]^{1/n} .
		$$
	Hence, by Proposition 2.7 in \cite{TVH} and  Theorem 1.2 in \cite{CL22} we obtain that
	\begin{align*}
		\int_\Omega (dd^c \max(  w  + v , k \psi ))^n 
		& \leq \limsup_{j\to+\infty} \int_{\Omega_j} (dd^c \max( \hat w_j +\hat v_j, k \psi_j))^n 
		\\ &\leq \limsup_{j\to+\infty} \int_{\Omega_j}  (dd^c  ( \hat w_j +\hat v_j ))^n 
		\\ &\leq \limsup_{j\to+\infty}\left( \left[ \int_{\Omega_j}  (dd^c    \hat w_j  )^n  \right]^{1/n}
		+  \left[  \int_{\Omega_j}  (dd^c  \hat v_j  )^n \right]^{1/n}\right)^n 
		\\ &\leq \left( \left[ \int_{\Omega}  (dd^c    \max( w ,-1)  )^n  \right]^{1/n}
		+  \left[  \int_{\Omega}  (dd^c  \max(  v ,-1)  )^n \right]^{1/n}\right)^n .
	\end{align*}
	This implies that 
\begin{align*}
	&\sup_{k\geq 1} \int_\Omega (dd^c \max(  w  +  v , k \psi ))^n
	\\&  \leq \left( \left[ \int_{\Omega}  (dd^c    \max( w ,-1)  )^n  \right]^{1/n}
		+  \left[  \int_{\Omega}  (dd^c  \max(  v ,-1)  )^n \right]^{1/n}\right)^n <+\infty.
\end{align*}
	Therefore, $w+v\in \mathcal F(\Omega)$, and thus, we conclude by the hypotheses \eqref{e221} that $u\in \mathcal F(\Omega)$. 
	
	Now,  it is easy to see that 
	$$
		\hat w_j+ \hat v_j \leq \hat  u_j \leq \hat v_j \text{ in } \Omega_j.
	$$
	By Lemma 4.1 in \cite{ACCH} and Theorem 1.2 in \cite{CL22} we have 
	\begin{equation}
		\label{e223}
	\begin{split}
		P(dd^c v)^n 
		& = 1_{\Omega_j \cap \{\hat v_j=-\infty\}} (dd^c \hat v_j)^n 
		\\& \leq 1_{\Omega_j \cap \{\hat u_j=-\infty\}} (dd^c \hat u_j)^n =P(dd^c u)^n 
		\\& \leq  1_{\Omega_j \cap \{\hat v_j  =-\infty\}} (dd^c (\hat w_j +\hat v_j))^n  .
	\end{split}
	\end{equation} 
	Since $ w \in \mathcal F^a(\Omega)$ so  $\hat w_j \in \mathcal F^a(\Omega_j)$ and hence, we infer by Lemma 4.12 in\cite{ACCH} that 
		\begin{align*} 
		1_{\Omega_j \cap \{\hat v_j  =-\infty\}} (dd^c (\hat w_j +\hat v_j))^n  =1_{\Omega_j \cap \{\hat v_j  =-\infty\}} (dd^c  \hat v_j) ^n =P(dd^c v)^n.
	\end{align*}
	Combining this with \eqref{e223} we obtain that 
	$$
	P(dd^c u)^n=P(dd^c v)^n.
	$$
	The proof is complete. 
\end{proof}

\begin{lemma}
\label{le31}
	Let $\Omega$ be a bounded $\mathcal F$-hyperconvex domain in $\mathbb C^n$ that has the $\mathcal F$-approximation property. Assume that  $u\in \mathcal F(\Omega)$ and define
	$$
	v_j:=\sup\{ \psi \in \mathcal F \text{-} PSH^-(\Omega): \psi \leq u+j  \text{ on } \Omega \} , \ j\in \mathbb N.
	$$
	Then,    the $\mathcal F$-upper semi-continuous majorant $v$ of $\sup_{j\geq 1} v_j $	in  $\Omega$ belongs to the class $\mathcal F(\Omega)$ and satisfies 
	$$
	\begin{cases}
		P(dd^c v)^n=P(dd^c u)^n &\text{ in } \mathbb C^n,
		\\ NP(dd^c v)^n=0 &\text{ on } QB(\Omega).
	\end{cases} 	
	$$
\end{lemma}

\begin{proof}
	Since $\Omega$ is  $\mathcal F$-open, so $v_j$ are $\mathcal F$-plurisubharmonic functions in $\Omega$, and hence, $v_j \in\mathcal F(\Omega)$ because 
	$$u \leq v_j\leq 0 \text{ in } \Omega.$$
	This implies that $v\in \mathcal F(\Omega)$.
	By Theorem 1.2 in \cite{HC18a}, we can find sequences of plurisubharmonic functions $\varphi_{j,s}$ defined on  bounded hyperconvex domains  $\Omega_s$ such that 
	$$
	\Omega_s \supset \Omega_{s+1} \supset \Omega=\bigcap_{s=1}^{+\infty} \Omega_s
	$$ 
	and 
	$$\varphi_{j,s} \nearrow v_j \text{ a.e. in } \Omega \text{ as } s\nearrow+\infty .
	$$
	
	\noindent Now,  Lemma 5.14 in \cite{Ce2} tells us that there exists a sequence $\{g_j\}\subset \mathcal F^a(\Omega_1)$  satisfying 
	$$
	(dd^c g_j)^n =1_{\Omega \cap \{-\infty<u<-j+1\}} (dd^c u)^n \text{ on }  \Omega_1. 
	$$
	Obviously that the measure sequence $\{(dd^c g_{j})^n\}$ is  decreasing  and converges to $0$ in $\Omega_1$. Hence, Theorem 5.15  in \cite{Ce2} shows that 
	$$
	g_j \nearrow 0 \text{ a.e. in } \Omega_1.
	$$

\noindent 
	Since
	\begin{align*}
	f_{j,s}:&=\sup\{\varphi \in PSH^-(\Omega_s): \varphi \leq v_j \text{ on } \Omega\}
	\\&=\sup\{\varphi \in PSH^-(\Omega_s): \varphi \leq u+j \text{ on } \Omega\},
	\end{align*}
	Theorem 1.2 in \cite{CL22} tells us that  $f_{j,k}\in \mathcal F(\Omega_k)$ and  
	\begin{equation}
		\label{e32}
		P(dd^c u)^n   = 1_{\Omega_s \cap \{f_{j,s}=-\infty\}} (dd^c f_{j,s})^n \text{ in } \mathbb C^n.
	\end{equation} 
	Let $ j\geq k$ be positive integer numbers. Since $f_{k,s} \leq f_{j,s}$ on $\Omega_s$, the inequality \eqref{e32} shows that $f_{k,s}  \in \mathcal N^a(\Omega_s, f_{j,k})$. Hence, using Theorem 1.1 in \cite{CL22} and Therem 1.2 in \cite{CL22} we infer by \eqref{e32} that 
\begin{align*}
		(dd^c (f_{j,s}+g_k))^n 
	& \geq (dd^c f_{j,s})^n +(dd^c g_{k})^n 
	\\& \geq P (dd^c u)^n +1_{\Omega \cap \{-\infty<u<-k+1\}} (dd^c u)^n
	\\& \geq (dd^c f_{k,s})^n.
\end{align*}
	Therefore, Proposition 2.2 in \cite{HTT17} that 
	$$f_{j,s}+g_k \leq f_{k,s} \text{ on } \Omega_s.$$
	This implies that 
	$$
	\varphi_{j,s} + g_k \leq f_{j,s}+g_k \leq f_{k,s} \leq v_k \text{ on } \Omega.
	$$
	Letting $s\to+\infty$, we obtain that 
	$$
		v_j + g_k \leq   v_k \text{ on } \Omega, \ \forall j\geq k\geq 1,
	$$
	and thus, 
	\begin{equation}\label{e33}
	v + g_k \leq   v_k \text{ on } \Omega,\ \forall  k\geq 1.
	\end{equation}
	We set 
	\begin{align*}
	f:=\sup\{\varphi \in PSH^-(\Omega_1): \varphi \leq v  \text{ on } \Omega\} .
	\end{align*}
	Since $v\geq v_k $ in $\Omega$, we infer by \eqref{e33} that
	$$
	f+g_k \leq f_{k,1} \leq f \text{ on } \Omega_1.
	$$
	Hence,  Theorem 1.2 in \cite{CL22} implies that 
		\begin{align*}
			P(dd^c v)^n & =1_{\Omega_1\cap \{f =-\infty\}} (dd^c f )^n
			\\& = 1_{\Omega_1 \cap \{f_{k,1}=-\infty\}} (dd^c f_{k,1})^n
		\\& =P(dd^c u)^n.
		\end{align*}	
	On the other hand, Theorem 1.1 in \cite{CL22} tells us that 
	$$
	NP(dd^c v_j)^n =0 \text{ on } \Omega \cap \{u>-k\}, \ \forall j\geq k\geq 1.
	$$
	Letting $j\to+\infty$ we conclude by Theorem 4.5 in \cite{KS14} that 
		$$
	NP(dd^c v)^n =0 \text{ on } \Omega,
	$$
	and thus, Lemma is proved. 
	\end{proof}

We now able to give the proof of Theorem \ref{th41}.

\begin{proof}[Proof of Theorem \ref{th41}]
Let $\hat \Omega \supset \Omega$ be a bounded hyperconvex domain in $\mathbb C^n$.  For  $f\in \mathcal F(\Omega)$, we set 
   $$
 \hat f:= \sup\{g \in PSH^-(\hat\Omega): g\leq f \text{ on } \Omega\}.
 $$
 Theorem 1.2 in \cite{CL22} tells us that $\hat f \in  \mathcal F(\hat\Omega)$ and  
\begin{equation}
	\label{e36}
	 P(dd^c f)^n=1_{\hat  \Omega \cap \{\hat f=-\infty\}} (dd^c \hat f)^n \  \text{ in } \mathbb C^n.
\end{equation}
The proof is split into three steps.

{\em Step 1.} We prove that 
there exists $g \in \mathcal F( \Omega)$ such that  	
 	\begin{equation}
	\label{e3113}
		\begin{cases}
		P(dd^c g)^n=\mu  &\text{ in } \mathbb C^n,
		\\ NP(dd^c g)^n=0 &\text{ on } QB(\Omega).
	\end{cases} 
 	\end{equation}
Indeed, by the hypotheses, we infer by \eqref{e36}  that  
 $$
 \mu \leq P(dd^c w)^n \leq   (dd^c \hat w)^n \text{ in } \hat \Omega.
 $$
 Theorem 4.14 in \cite{ACCH} tells us that there exists a function $h\in \mathcal F(\hat \Omega)$ such that % $h \geq \hat w$ and 
 \begin{equation}
	\label{e35}
	 (dd^ch)^n=  \mu  \  \text{ on } \hat \Omega.
\end{equation} 
 	Therefore, Theorem \ref{th21} states that 
 	$$v:= h |_\Omega \in \mathcal E(\Omega).$$
 	Let $\varphi \in \mathcal E_0(\Omega)$. By Lemma \ref{le30'} we can find a decreasing sequence $\{v_j\} \subset \mathcal F(\Omega)$ such that 
	$$
		v_j =v \text{ on } \Omega \cap \{j \varphi<-1\}.
	$$
	It follows that 
 \begin{equation}
	\label{e37}
	 h \leq \hat v_{j+1} \leq \hat v_j  \text{ on } \hat\Omega
\end{equation}  	and
 \begin{equation}
	\label{e38}
	 h=\hat v_j \text{ on } \Omega \cap \{j\varphi<-1\}.
\end{equation} 
		Now, by Lemma \ref{le30} we can find plurisubharmonic functions $\varphi_{j,1}$ and $\varphi_{j,2}$ defined on Euclidean neighborhood $\Omega'$ of $\Omega$ such that 
	$$
	 \Omega \cap \{j\varphi<-1\}=\Omega' \cap \{\varphi_{j,1}<\varphi_{j,2}\}.
	$$
	Hence, using Theorem 1.1 in \cite{HHiep16} we deduce by \eqref{e38} that 
 \begin{equation}
	\label{e39}
	(dd^c h)^n=(dd^c \hat v_j)^n \text{ on } \Omega \cap \{j\varphi<-1\}.
	\end{equation}  	
	On the other hand,  we infer by \eqref{e37} and  Lemma 4.1 in \cite{ACCH}  that 
\begin{align*}
	1_{  \{\hat v_j =-\infty\}} (dd^c \hat v_j)^n \leq  
 	1_{  \{h=-\infty\}} (dd^c h)^n \  \text{ in } \hat \Omega.
\end{align*}
 	Combining this with \eqref{e36}, \eqref{e35} and \eqref{e39} that  
 	\begin{equation}
	\label{e311}
	 1_{\Omega \cap \{j\varphi<-1\}} \mu \leq  P(dd^c   v_j)^n \leq 
 	 \mu  \  \text{ in } \hat \Omega.
\end{equation}
 	
 	\noindent 
 	We set 
	$$
	g_{j,k}:=\sup\{ \psi \in \mathcal F \text{-} PSH^-(\Omega): \psi \leq v_j+k \text{ on } \Omega \} , \ k\in \mathbb N.
	$$
	Let $g_j$ be   the $\mathcal F$-upper semi-continuous majorant   of $\sup_{k\geq 1} g_{j,k} $ 	in $\Omega$. Lemma \ref{le31} tells us that  $g_j \in \mathcal F(\Omega)$ and 
\begin{equation}
	\label{e3112}
		\begin{cases}
		P(dd^c g_j)^n=P(dd^c v_j)^n &\text{ in } \mathbb C^n,
		\\ NP(dd^c g_j)^n=0 &\text{ on } QB(\Omega).
	\end{cases} 
\end{equation}
Since $\{v_j\}$ is decreasing, so 
$$ g_{j+1,k} \leq  g_{j,k} \text{ on } \Omega,$$
and thus, $\{g_j\}$ is a decreasing sequence.
Moreover, using Theorem 1.3 in \cite{CL22} we obtain by \eqref{e3112} that 
	\begin{align*}
		\sup_{j\geq 1} \int_\Omega (dd^c \max(g_j,-1))^n 
		& = \sup_{j\geq 1} \int_\Omega P(dd^c g_j)^n 
		\\& \leq   \int_\Omega d\mu \leq   \int_\Omega P(dd^c w)^n <+\infty. 	
	\end{align*}
 	Hence, Lemma \ref{le27} tells us that 
 	$$g:=\lim_{j\to+\infty} g_j  \in \mathcal F( \Omega).$$
 	Since $\Omega \cap \{j\varphi<-1\} \nearrow \Omega$ as $j\nearrow +\infty$, we conclude by \eqref{e39}, \eqref{e311} and \eqref{e3112} that  	
$$
		\begin{cases}
		P(dd^c g)^n=\mu  &\text{ in } \mathbb C^n,
		\\ NP(dd^c g)^n=0 &\text{ on } QB(\Omega).
	\end{cases} 
$$
 	This proves step 1.

 	{\em Step 2.}  	We prove that there exist $w_1\in\mathcal F^a( \Omega)$ and $u\in \mathcal F\text{-} PSH^-(\Omega)$ such that 
 	 	\begin{equation}
	\label{e334'}
		\begin{cases}
		NP(dd^c u)^n =\nu  &\text{ in } QB(\Omega),
		\\ w_1+g\leq u \leq g & \text{ on } \Omega.
	\end{cases}
 	\end{equation}
 	Indeed, by the definition of $g_j$ we have 
 	$$g_j\geq g_{j,k}\geq v_j \geq h \text{ on } \Omega,\ \forall j,k\geq 1,$$
 	and hence,   
 	$$
 	g=\lim_{j\to+\infty} g_j \geq h \text{ in } \Omega.
 	$$
 	We set 
 	$$U_j:=\Omega \cap \{h>-j\}.$$
 	Since $h$ is plurisubharmonic function in $\hat\Omega$, so $U_j$ is $\mathcal F$-open. From $g$ is bounded on $U_j$, we infer by \eqref{e3113} and Theorem 1 in \cite{HHV17} that $g$ is $\mathcal F$-maximal on $U_j$. 
 	Therefore,  the proof of Theorem 1.1 in \cite{Hong17} tells us that 
 	$$
 	\psi_j:=\sup \{ f\in \mathcal F\text{-} PSH^-(U_j): f\leq g \text{ on } U_j \text{ and }  NP(dd^c f)^n\geq \nu \text{ on } QB (U_j)\}
 	$$
 	is $\mathcal F$-plurisubharmonnic in $U_j$ and satisfies 
 	\begin{equation}
 		\label{e331}
 		NP(dd^c \psi_j)^n =  \nu \text{ on } QB (U_j).
 	\end{equation}
 	Let $w_1\in \mathcal F^a(\Omega)$ be such that 
	$$
	(dd^c w_1)^n = NP(dd^c w)^n \text{ on } \Omega. 
	$$ 
	Since  $U_j\subset U_{j+1}$ and $NP(dd^c (w_1+g))^n \geq \nu$ in $QB(U_j)$, we deduce by the definition of  $\{\psi_j\}$  that
 	\begin{equation}
 		\label{e332}
 		w_1+g\leq \psi_{j+1} \leq \psi_j \leq g \text{ on } U_j.
 	\end{equation}
 	%.  
  	Hence, the $\mathcal F$-upper semi-continuous majorant $\psi$ of  
 	$$
 	\limsup_{j\to +\infty } \psi_j (z), \ z\in \Omega \cap\{h>-\infty\}
 	$$ 
 	is a finite $\mathcal F$-plurisubharmonic function on $\Omega \cap \{h>- \infty\}$. Using Theorem 4.5 in \cite{KS14} we deduce by \eqref{e331} and \eqref{e332} that 
 	\begin{equation}
	\label{e333}
		\begin{cases}
		NP(dd^c \psi )^n=\nu  &\text{ on } QB(\Omega \cap \{h>- \infty\}),
		\\ w_1+g\leq \psi \leq g & \text{ on } \Omega \cap \{h>- \infty\}.
	\end{cases} 
 	\end{equation}
	Now, using Theorem 3.7 in \cite{KFW11} we obtain that  
	$$
	u(z):= 
	\begin{cases}
		\psi (z) & \text{ if } z\in \Omega \cap \{h>- \infty\}
		\\ \mathcal F \text{-} \lim_{\Omega \cap \{h>- \infty\} \ni \xi \to z} \psi(\xi)& \text{ if } z\in \Omega \cap \{h=- \infty\}
	\end{cases}
	$$
	is a $\mathcal F$-plurisubharmonic function in $\Omega$. Since $\Omega\cap \{h=-\infty\}$ is a pluripolar set, we deduce by \eqref{e333} and Corollary 3.2 in \cite{MW10} that 
$$
		w_1+g\leq u \leq g  \text{ on } \Omega.
$$
Combining this with \eqref{e333} we conclude that \eqref{e334'} has been proven. 

{\em Step 3.}  The inequality  \eqref{e334'}  tells us that
$$
w_1+g\leq u\leq g \text{ in } \Omega.
$$
Lemma \ref{le28} states that 
$u\in\mathcal F(\Omega)$ and 
$$
P(dd^c u)^n = P(dd^c g)^n\text{ in } \mathbb C^n.
$$
Therefore, we conclude by  \eqref{e3113} and  \eqref{e334'}  that 
$$
\begin{cases}
		P(dd^c u)^n = P(dd^c g)^n=\mu  &\text{ in } \mathbb C^n,
		\\ NP(dd^c u)^n=\nu  &\text{ on } QB(\Omega).
	\end{cases} 
$$
	The proof is complete. 
\end{proof}

\end{document}